\documentclass[12pt]{article}

\usepackage{amssymb}
\usepackage{amsmath}
\usepackage{xcolor}

\newcommand{\Set}{\mathsf{Set}}

\newcommand{\id}{\mathsf{id}}
\newcommand{\C}{\mathcal{C}}

\newcommand{\Top}{\mathsf{Top}}
\newcommand{\Haus}{\mathsf{Haus}}

\newcommand{\Tych}{\mathsf{Tych}}

\newcommand{\qed}{\mbox{}\hfill\mbox{}%
\mbox{$\square$}\medbreak}

\newcommand{\M}{{\mathcal M}}
\newcommand{\calC}{{\mathcal C}}

\newcommand{\calP}{{\mathcal P}}
\newcommand{\calA}{{\mathcal A}}

\newcommand{\calS}{{\mathcal S}}
\newcommand{\calQ}{{\mathcal Q}}

\newtheorem{theorem}{Theorem}[section]
\newtheorem{definition}[theorem]{Definition}
\newtheorem{proposition}[theorem]{Proposition}
\newtheorem{lemma}[theorem]{Lemma}
\newtheorem{remark}[theorem]{Remark}
\newtheorem{corollary}[theorem]{Corollary}
\newtheorem{example}[theorem]{Example}

\title{Smallness in Topology}
\author{Ji$\check{\text{r}}$\'{i} Ad\'{a}mek\footnote{ The first- and third-named authors acknowledge the support by the Grant Agency of the Czech Republic under the grant 22-02964S.},\; Miroslav Hu$\check{\text{s}}$ek,\, Ji$\check{\text{r}}$\'{i} Rosick\'{y}$^*$\!, Walter Tholen\footnote{The fourth-named author acknowledges partial financial assistance by the Natural Sciences and Engineering Council of Canada under the Discovery Grants Program, no. 501260.}}
\date{
{\em In memory of our friend Horst Herrlich}}
 
\begin{document}
\maketitle

\begin{abstract}

Quillen's notion of small object and the Gabriel-Ulmer notion of finitely presentable or generated object are fundamental in homotopy theory and categorical algebra. Do these notions always lead to rather uninteresting classes of objects in categories of topological spaces, such as all finite discrete spaces, or just the empty space, as the examples and remarks in the existing literature may suggest?
 
This article demonstrates that the establishment of full characterizations of these notions (and some natural variations thereof) in many familiar categories of spaces can be quite challenging and may lead to unexpected surprises. In fact, we show that there are significant differences in this regard even amongst the categories defined by the standard separation axioms, with the T$_1$-separation condition standing out. The findings about these specific categories lead us to insights also when considering rather arbitrary full reflective subcategories of the category all topological spaces.

\end{abstract}

\medskip
{\it Keywords:} finitely presentable object, finitely generated object, finitely small object, directed colimit, Hausdorff space, T$_0$-space, T$_1$-space, compact space.

{\it Mathematics Subject Classification:} 54B99, 18B30, 54D10.

\section{\small Introduction}

Small objects in a category $\C$ were introduced in homotopy theory
by Quillen \cite{Q}. Given a class $\mathcal M$ of morphisms,
an object $X$ is \emph{finitely small} (aka $\aleph_0$-small) w.r.t. $\mathcal M$, if its
hom-functor $\mathcal C(X,-):\mathcal C\to\Set$ preserves colimits of
\emph{continuous chains} of
morphisms from $\mathcal M$. Continuity means that the chain $D$ is indexed by a limit
ordinal $\alpha$ (the linearly ordered set of all ordinals $i< \alpha$), and
 for every limit ordinal $j< \alpha$ the step $j$ of the chain is given by the
colimit of the subchain indexed by all $i<j$; in other words, the given functor $D:  \alpha \to \C$ preserves colimits. More generally, $X$ is
\emph{$\lambda$-small} w.r.t. $\M$ if the hom-functor
preserves the above colimits for all ordinals of cofinality al least $\lambda$.

Shortly afterwards Gabriel and Ulmer \cite{GU1971} introduced the concept
of a \emph {finitely presentable} and \emph {finitely generated} object:
here the hom-functor is required to preserve all directed colimits,
or directed colimits of monomorphisms, respectively. Directed colimits can be
reduced to colimits of continuous chains in general, see Theorem \ref{smooth}.
Thus finitely presentable objects are precisely the finitely small ones
if $\M$ consists of all morphisms. What about other classes $\M$?
In the present paper we study (finite) smallness in the category $\Top$
and its subcategories, {\it e.g.} in the categories $\Haus$ of Hausdorff spaces, $\Top_1$ of T$_1$-spaces, and $\Top_0$ of T$_0$-spaces.
The concept of finitely generated or finitely small space then depends
fundamentally on the category in question.

Besides the class $\M$ of all
monomorphisms (injective continuous maps), we consider the classes of all
subspace embeddings, and of all open embeddings. In a general category we
call an object \emph{finitely generated w.r.t. $\M$} if its hom-functor
preserves colimits of directed diagrams with connecting maps
in $\M$.

We shall see that
for almost all of our results, with one surprising exception, there is
no difference between the use of directed colimits and colimits of continuous chains.
In fact, if $\M$ consists of all monomorphisms, or of all open embeddings,
directed colimits can always be reduced to colimits of chains (Corollary \ref{C:smooth}).
So, the only case where finite smallness may differ from finite generation
(for the classes we study) is that of embeddings.

We  also look at full subcategories of $\Top$, especially at some of its
prominent reflective subcategories that were categorically presented by
Horst Herrlich in his early land-mark monograph \cite{Herrlich1968}. We
characterize their finitely generated and finitely small objects.
Whereas Sections 7 and 8 deal with rather general full subcategories, 
we first concentrate on $\Haus$, $\Top_1$ and $\Top_0$. Let us mention
these results in greater detail.

\subsection {{\small Small topological spaces}}

It is well known that the finitely generated objects of the
category $\Top$ of topological spaces are precisely the finite discrete
spaces, and that they are all even finitely presentable. The original proof due to Gabriel and Ulmer  \cite{GU1971}, 6.4,  has a shortcoming
---the only mathematical error in their monograph of which the authors of this paper are aware. Fortunately though, as we show in detail in Section 2, their argumentation may be amended to yield a correct proof, in such a way that the elegance and generality of its original does not get compromised. 

An apparent, but known to be faulty, simplification of the argumentation given in  \cite{GU1971} appeared as Example 1.2(10) of the book \cite{AR1994}. Also the argument given at the beginning of Section 2.4 of Hovey's monograph \cite{Hovey1999}, for the much narrower claim that a non-discrete two-point space fails to be finitely presentable, is known to be faulty; see the relevant entry \cite{nLab} in the nLab and the discussion blog \cite{mathoverflow}.

We characterize finitely generated objects w.r.t. open embeddings as precisely the compact
spaces, and finitely presentable objects w.r.t. embeddings as precisely the
finite spaces (Theorem \ref{T}) .

Using continuous chains or directed diagrams makes no difference for all of the above
results in $\Top$.

\subsection {{\small Small Hausdorff and Tychonoff spaces}}

Let us consider smallness in $\Haus$ and $\Tych$, the full subcategories of
$\Top$ of Hausdorff and of Tychonoff spaces, respectively. Of great
assistance is Herrlich's example \cite{Herrlich1969} of a non-Hausdorff
space that is the union of an $\omega$-chain of closed Tychonoff subspaces.
It is  beautifully crafted in a minimalistic
fashion that is typical for many of Herrlich's works, and we present it with a detailed proof in \ref{Herrlichex} and \ref{proof}. This example allows us
to conclude that, in $\Haus$ and $\Tych$, the empty space is the only object which is finitely
generated w.r.t. closed embeddings (Corollary \ref{T0fp}). By contrast, the
objects finitely generated w.r.t. open embeddings are again the compact spaces
(Corollary \ref{C:compact}).

\subsection {{\small Small T$_1$-spaces}}
Concerning smallness, the category $\Top_1$ is the most colorful one of the full
subcategories of $\Top$ we consider. On the one hand, no non-empty
space is finitely presentable (Theorem \ref{T0fp}). To prove this
we give an example of a directed collection of
quotients of the unit interval whose colimit in $\Top$ is a two-point
indiscrete space. (A more complicated example of this sort is due to Dugundji
\cite{Dugundij1966}, p. 422.)

On the other hand, the finitely generated objects of $\Top_1$ are the finite (and, hence, discrete)
spaces (Theorem \ref{T1fg}), and the finitely generated spaces
w.r.t. open embeddings are precisely the compact spaces
(Corollary \ref{C:compact}). Once again, directed colimits
play the same role as colimits of continuous chains in all of these four
results.

By contrast, for the class $\M$ of all embeddings we get two different
smallness results (see Theorem \ref{T1emb}): A space in the category $\Top_1$ is

\begin {itemize}
\item[$\mathrm{(1)}$] Finitely generated w.r.t. embeddings if, and only if, it is
finite.
\item[$\mathrm{(2)}$] Finitely small w.r.t. embeddings if, and only if, it is
countable and compact.
\end {itemize}
This makes $\Top_1$ the unique exception mentioned
earlier.

\subsection {{\small Small T$_0$-spaces}}

It turns out that a characterization of finitely presentable objects and
finitely generated ones (in the variants we consider) does not get easier
at all in categories of spaces with low separation conditions.
In \ref{T0example} we present an example of an $\omega$-chain of monomorphisms between T$_0$-spaces
whose colimit in $\Top$ nevertheless has an indiscrete two-point subspace. Using this example, we
are able to show that the empty space is the only finitely generated object
in $\Top_0$ (Theorem \ref{T0fg}).

We also prove that finite T$_0$-spaces are precisely the objects finitely
generated w.r.t. embeddings. The proof is not easy, it uses substantially
Ramsey's Theorem.

\medskip

Here is a quick summary  of the main results proved in this paper:

\begin{table}[h!]
	\begin{tabular}{c|c|c|c|c|c}
		cat. & fin.\ pres. & fin.\ gener. & f.g. emb. & f. small emb. & f.g.op.emb.\\ \hline
		$\Top_{\phantom{0}}$ & finite discr. & finite discr. & finite & finite & compact\\ \hline
		$\Top_0$ \label{key}& empty & empty & finite  & finite & compact\\ \hline
		$\Top_1$ & empty & finite discr.& finite discr.& ctbl., comp. &  compact\\ \hline
		$\Haus$ & empty & empty & empty & empty & compact
	\end{tabular}
\end{table}

\section{\small The corrected Gabriel-Ulmer proof}

We now give a detailed proof that finitely presentable (or finitely generated) topological spaces are 
precisely the finite discrete spaces. In the next section we prove that \emph{all} finite topological spaces are
finitely generated w.r.t. embeddings. We recall the relevant definitions. A directed poset is a poset with upper bounds for all finite subsets.

\begin{definition}	
	\rm	
	(1) An object $X$ of a category $\C$ is {\it finitely presentable} if, and only if, its hom-functor 
	preserves directed colimits. In detail: let $D=(Z_i)_{i \in I}$ be a directed diagram  (indexed by a directed poset $I$) with a colimit
	$c_i: Z_i \to Z$ in $\C$. Then every morphism $f:X \to Z$ factorizes, for some $i \in I$, \emph{essentially uniquely} through $c_i$:
\begin{itemize}	
\item[(i)] there is a morphism $g: X \to Z_i$ with $f=c_i \cdot g$, and 
	
\item[(ii)] if $g': X \to Z_i$ also fulfills $f=c_i \cdot g'$, then some connecting morphism $z_{i,j}: Z_i \to Z_j$ 
	merges $g$ and $g'$.
	\end{itemize}

\indent{(2)} An object $X$ is {\em finitely generated} if the above is only required to hold for directed diagrams of monomorphisms, {\em i.e.}, for diagrams whose connecting morphisms of $D$ are monic.	

(3) Given a class $\mathcal M$ of monomorphisms, $X$ is {\em finitely generated w.r.t. $\mathcal {M}$} (or {\em finitely small w.r.t. $\mathcal M$}) if the above is only required to hold for directed diagrams  of monomorphisms from  $\mathcal M$ (or just continuous chains of monomorphisms from  $\mathcal M$, resp.).

\end{definition}

\begin{remark}
\rm
	Condition (ii) can be strengthened as follows:
	
	\begin{itemize}
	\item[(iii)] For every $k \geq i$ and every pair $h,h': X \to Z_k$ of morphisms
	with $f=c_k \cdot h= c_k\cdot h'$, some connecting morphism $z_{k,j}: Z_k \to Z_j$ merges $h$ and $h'$.
	\end{itemize}
	
	Indeed, the directed subdiagram $D_k$ of $D$ indexed be the upper set of $k$
	has the same colimit $Z$ with the colimit cocone of all $c_l$ where $l \geq k$. For each $l$
	we have $c_l = c_k \cdot z_{k,l}$ which implies the equality
	$$c_l \cdot (z_{k,l} \cdot h)= c_l \cdot (z_{k,l} \cdot h').$$
	Applying (i) and (ii) to $D_k$, we conclude that there exists $l\geq k$ such that
	the last  equality implies that some connecting morphism $z_{l,j}$ merges
	$z_{k,l} \cdot h$ and $z_{k,l} \cdot h'$. The desired connecting map is $z_{k,j}$: since
	$z_{k,j} = z_{l,j} \cdot z_{k,l}$ we get
		$$z_{k,j} \cdot h = z_{l,j} \cdot z_{k,l} \cdot h =z_{l,j} \cdot z_{k,l} \cdot h' = z_{k,j} \cdot h'.$$

\end{remark}
\begin{proposition}\label{fpTop}
The finitely generated spaces in the category $\Top$ are precisely the finite discrete spaces, and they are all finitely presentable.

\end{proposition}

\begin{proof} Quite trivially, if $X\in\Top$ is finite and discrete, $X$ is finitely presentable. Indeed, given a map $f:X\to Z\cong\mathrm{colim}_{i\in I}Z_i$ into the colimit of a directed system of spaces $Z_i$, each of the finitely many elements $f(x), x\in X,$ lies in the image of $Z_i\to Z$ for some $i\in I$ and, hence, for a common $i\in I$. The resulting factorization of $f$ through $Z_i$ is essentially unique already at the level of sets.

Equally trivially one sees that a finitely generated space $X$ must be finite, as follows. For the filtered system of its finite subsets $E$ considered as {\em indiscrete} spaces, connected by inclusion maps, their colimit $Z$ in $\Top$ may be taken to have the same underlying set as the space $X$, but with open sets all those $V\subseteq X$ for which $V\cap E =\emptyset$ or $E\subseteq V$, for every finite subset $E\subseteq X$. But since for any non-empty and proper subset $W\subseteq X$ one would find a two-point set $D\subseteq X$ with $W\cap D\neq\emptyset$ and $ D\nsubseteq W$, the topology of $Z$ is indiscrete. That makes the identity map $X\to Z$ continuous and, hence, factorize through a finite subset of $X$.

Less trivially, in order to show that $X$ must be discrete, let us assume that we had a non-open subset $S$ in $X$. For the ordered set $\omega$ of natural numbers and every finite subset $F\subseteq \omega$, we form the topological space $Y_F$; its underlying set is the disjoint union $X+\omega$, and its open sets are, other than $\emptyset$ and $X+\omega$,  those subsets $U$ of $X+\omega$ which meet $X$ in precisely $S$, and meet $\omega$ in some non-empty end of $\omega$ disjoint from $F$; that is:
$$ \text{1. } U\cap X=S,\;\; \text{ 2. } U\cap F=\emptyset,\;\;\text{ 3. } U\cap\omega=\{n\in\omega\mid n\geq m\} \text{ for some }m\in \omega.$$
This defines a topology on $X+\omega$, and as above one has that the colimit $Y=\mathrm{colim}_{F}Y_F$ may be taken to have underlying set $X+\omega$ provided with the indiscrete topology; indeed, in order to satisfy condition 2. for all $F$, a non-trivial open set $U$ in $Y$ would be forced to lie in $X$, in contradiction to condition 3.
Consequently, the identity map $X\to Y$ is continuous, but it cannot factor continuously through any space $Y_F$, since $Y_F$ contains an open set of the form $S+\{n\mid n\geq m\}$, the intersection of which with $X$ fails to be open---a contradiction.
\end{proof}
\qed

\begin{remark}
\rm
The non-trivial part of the above proof follows the argumentation given in \cite{GU1971}, page 65, with one essential difference: when we transcribe the Gabriel-Ulmer proof into the above context and notation, then it misses condition 3 for the non-trivial open sets in  $Y_F$. But this omission would allow the set $S$ to be open in every space $Y_F$ and, hence, in their colimit $Y$, which would make the remainder of the argumentation collapse. 
\end{remark}

In the proof of the Proposition, let us now replace the countable set $\omega$ by any infinite regular cardinal $\lambda$, provided with its order as an ordinal number, and let us read 
`finite' as `of cardinality less than $\lambda$'. Then the proof remains intact, and we have indeed verified the original Gabriel-Ulmer claim:

\begin{theorem}\label{Topthm}
For every infinite regular cardinal $\lambda$, the $\lambda$-generated objects in $\Top$ are precisely the discrete spaces of cardinality less than $\lambda$, and they are all $\lambda$-presentable.
\end{theorem}

\section{\small Smallness in $\Top$}

We characterize smallness w.r.t. embeddings and open embeddings for objects in the category $\Top$
of topological spaces. We also present some results about full subcategories of $\Top$.

Let us start with a theorem showing that directed colimits and continuous chains are often equally `strong'
in every category $\mathcal C$ with directed colimits.

Let $\mathcal M$ be a class of morphisms containing all
isomorphisms and closed under composition. Then $\mathcal M$ can be considered
as a generally non-full subcategory of $\mathcal C$ with the same objects.
To say that the subcategory $\M$ is {\it closed under directed colimits} in $\C$ means
  that, for every directed diagram $D=(Z_i)_{i\in I}$ with all connecting morphisms in $\M$, the colimit cocone $c_i:Z_i\to {\rm colim}\,D \;(i\in I)$ in $\C$ has the following properties.
  \begin{itemize}
  	\item[(1)] All morphisms $c_i$ lie in $\M$.	
  	\item[(2)] Given any cocone $f_i:Z_i\to A\;(i\in I)$ of morphisms in $\M$, the unique factorization morphism $\mathrm{colim}\,D\to A$ lies in $\M$, too. 
  \end{itemize}

Analogously for $\M$ being closed under colimits of continuous chains.

\begin{theorem}\label{smooth} Let $\mathcal C$ be a category with directed colimits, and   $\mathcal M$ a class of morphisms that is closed under composition and contains all isomorphisms. Then,
if $\mathcal M$ is closed under colimits of continuous chains, $\M$ is closed under all directed colimits.
	
\end{theorem}
\begin{proof} 
By \cite{AR1994}, Lemma 1.6, for every directed poset of cardinality $\lambda$ there exists a continuous chain
	$I_k \subseteq I$ of directed subposets (for $k<\lambda$) of cardinalities
	$\mid I_k |<\lambda$ whose union is all of $I$; continuity means that, for every limit ordinal $k<\lambda$, 
	the poset $I_k$ is the union of all $I_j$ for $j<k$.

	Given a directed diagram $D$ of morphisms in $\M$ with a colimit
	$c_i: Z_i \to Z\; (i \in I)$, we prove properties (1) and (2) above. 
	We proceed by induction on the cardinality $\lambda$ of $I$.
	
	If $I$ is finite, then $I$ has a largest element, and the statement is trivial.
	
	Otherwise, assuming the statement holds for all diagrams of
	cardinalities smaller than $\lambda$, we prove that it holds for $D$.  Let $I_k\subseteq I$ be a continuous chain as above.
	
	For each $k<\lambda$ we form the subdiagram $D_k =(Z_i)_{i\in I_k}$
	of $D$ and its colimit
	$$c_{i,k}: Z_i \to Q_k$$
	for $i\in I_k$. By induction hypothesis, these morphisms lie in $\M$. Moreover, given $k<l<\lambda$ we have
	the unique factorization $q_{k,l}: Q_k \to Q_l$ with
	$$c_{i,l}=q_{k,l} \cdot c_{i.k}$$
	for $i\in I_k$, and again by induction hypothesis, it lies in $\M$. We obtain a chain of  objects $Q_k$ and morphisms
	 $q_{k,l}: Q_k \to Q_l (k\leq l \leq \lambda)$. This chain is continuous because of the continuity of the collection of all $I_k$.

	We further have, for each $k<\lambda$,
	the unique morphism $q_k: Q_k \to Z$ with
	$$c_i=q_k \cdot c_{i,k}$$
	for $i \in I_k$. These morphisms clearly form the colimit of the chain $Q_k (k<\lambda)$. Thus they lie in $\M$, following the closure of this class under colimits of continuous chains.
	This proves item (1) above: $c_i$ lies in $\M$ since this class is closed under composition.
	
We now verify item (2). Every cocone $f_i: Z_i \to X$ of $D$ consisting of members of $\mathcal M$
yields cocones of the subdiagrams $D_k$, and the unique factorizations $h_k:Q_k \to X$ with
$$h_k \cdot c_{i,k} = f_i \hspace{3mm} (i<k)$$
lie in $\M$ by induction hypothesis.
These morphisms (for all $k<\lambda$)
form a cocone of the chain of all $Q_k$. Indeed, given $l<k<\lambda$, the equality $h_k = h_l \cdot q_{k,l}$ follows since the cocone $c_{i,k} (i<k)$ is collectively epic, and we have
$$ h_k \cdot c_{i,k} =f_i= f_l \cdot c_{i,l}= h_l \cdot q_k \cdot c_{i,k}. $$
Thus we have a unique morphism $h :Z\to X$ with $h \cdot q_k = h_k$ for all $k<\lambda$.   Since $\M$ is closed under colimits of contuinuous chaions, $h$ lies in $\M$.
This is the desired factorization:
$$h \cdot c_i = h \cdot q_k \cdot c_{i,k} = h_k \cdot c_{i,k} = f_i.$$
\qed
\end{proof}

\medskip
 We now turn to classes of morphisms in $\Top$. By an \emph{embedding} we mean a monomorphism $m: X \to Y$ in $\Top$ such that $X$ caries the 
 \emph{initial topology}: every open set of $X$ has the form $m^{-1}(U)$ for some open set $U$ of $Y$.
 These are precisely the monomorphisms representing the subspaces of $Y$.
 Categorically, embeddings are precisely the regular monomorphisms in $\Top$.
 
 We prove that both monomorphisms and open embeddings form classes closed under directed colimits,
 but embeddings do not.  
 
 Let us recall that directed colimits of monomorphisms in $\Set$ may be given by directed unions of the corresponding subsets.

\begin{proposition}\label{P:directed} Both, the class of all monomorphisms and the class of all open embeddings, are closed under directed colimits in $\Top$.
\end{proposition}

\begin{proof}
(1) Monomorphisms: just use that the forgetful functor of $\Top$ preserves colimits (since it has a right adjoint),
and preserves and reflects monomorphisms.

(2) Open embeddings: Let $D = (Z_i)_{i \in I}$ be a directed diagram of open embeddings.
 Without loss of generality, we assume that, if $i \leq j$, then $Z_i$ is a subspace of $Z_j$, and the connecting map is the inclusion.
Then the colimit is the union $Z$, with open sets
those subsets $U$ with $U \cap Z_i$ open in $Z_i$ for each $i$. Therefore, every $c_i$ is an open mapping: every open set in $Z_i$ is also open in 
$Z_j$ for all $j>i$, thus it is also open in $Z$.

Given a cocone $f_i: Z_i \to X$ of open embeddings, the factorization map $f: Z \to X$ is clearly monic and continuous.
For every open set $U= \bigcup _{i \in I} U_i$ in $Z$ each set $f_i [U_i ]$ is open in $X$. The union of these sets is $f[U]$,
which is thus open, too. This proves that $f$ is an open mapping.
\qed
\end{proof}

\begin{corollary}\label{C:smooth} Let $\C$ be a full subcategory of $\Top$.

	{\rm(1)} If $\C$ is closed under monomorphisms and directed colimits of monomorphisms in $\Top$, then every finitely small
	object w.r.t. monomorphisms of $\C$ is finitely generated.

	{\rm(2)} If $\C$ is closed under open subspaces and colimits of open embeddings in $\Top$, then every finitely small
	object w.r.t. open embeddings in $\C$ is finitely generated w.r.t. open embeddings.	
\end{corollary}

Indeed, in (1) we conclude that monomorphisms on $\C$ form a class closed under directed colimits. Whenever the hom-functor of an object $C\in \C$ preserves colimits of continuous chains of monomorphisms in $\C$, then it preserves
directed colimits of chains. This follows from the reduction of the latter colimits to colimits of continuous chains that
we have seen in the proof of Theorem \ref{smooth}

The argument for open embeddings is the same.

\begin{proposition}\label{dir} Every directed diagram of embeddings in $\Top$ has a colimit cocone formed by embeddings.
	
\end{proposition}

\begin{proof}
Let $D = (Z_i)_{i \in I}$ be  a directed diagram of embeddings with a colimit cocone $c_i:Z_i \to Z$.
We prove that, if $I$ has a least element, $0$, then $c_0$ is an embedding.
The general statement then follows: each $c_i$ is an embedding
because our diagram has the same colimit as the subdiagram indexed by
the upper set of $i$ in $I$.

Our task is, given an open set $V \subseteq Z_0$, to find an open set $U \subseteq Z$ with $V=U \cap Z_0$.
The space $Z_0$ is a subspace of each $Z_i$, thus there exists an open
set in $Z_i$ that meets $Z_0$ in $V$. Denote by $U_i \subseteq Z_i$ the
largest such open set (the union of all open sets with the desired
property). Then, for all $i \leq j$ in $I$, we have $U_j \cap Z_i \subseteq U_i$, since $U_j \cap Z_i$ is open in $Z_i$ and also meets $Z_0$ in $V$.

This implies that the set $U = \bigcup _{j \in I} U_j$ fulfills $U \cap Z_i
= U_i$ for all $i$:
$$U \cap Z_i =\bigcup _{j\in I} (U_j \cap Z_j) \subseteq U_i$$
and the opposite inclusion follows from choosing $j:=i$.
. Thus $U$ is open in $Z$, and this is the desired subset with $V= U
\cap Z_0$.
\qed
\end{proof}

	\medskip
	
	For every infinite set $X$ the \emph{cofinite topology} has as non-empty open sets
	precisely those whose complement is finite. We denote the corresponding space  by $X_{\mathrm {cof}}$.
	
\begin{example}
	\rm
	 In contrast to the previous proposition, given an
	$\omega$-chain of embeddings in $\Top$ and a cocone of that chain formed by
	embeddings, the factorizing morphism from the colimit need not be an
	embedding.
	
	For every natural number $n$ consider the 
	space $A_n$ with the cofinite topology on the set
	$\{0,...,n-1\} \times \Bbb{N}$. Denote by $A$ the colimit of the
	$\omega$-chain of embeddings of $A_n$ for $n< \omega
	$. This is a space on the set  	$ \Bbb{N} \times \Bbb{N} $ in which the set $\Bbb{N} \times\{0\}$ 
	is closed (because its intersection with each $A_n$ is finite).
	 
	The inclusions of all $A_n$ into the
	 space 
	$( \Bbb{N} \times \Bbb{N}) _{\mathrm{cof}}$ form a cocone of embeddings. But the factorizing
	map  $\id:A \to ( \Bbb{N} \times \Bbb{N})_
{\mathrm{\mathrm{cof}}}$
 is not an embedding. Indeed, the set $\Bbb{N} \times\{0\}$ is not closed  in $( \Bbb{N} \times  \Bbb{N})_{\mathrm{cof}}$.
	
	\end{example}

\begin{corollary} \label{cor}
	In the categories $\Top$, $\Haus$, $\Top_1$ and $\Top_0$, finite generation is equivalent to finite smallness w.r.t. monomorphisms.
	And finite generation w.r.t. open embeddings is equivalent to finite smallness w.r.t. open embeddings.
\end{corollary}	

\medskip
Thus the only class of monomorphisms we study (in our leading examples of subcategories of $\Top$) where the two `smallness concepts' can differ
is that of embeddings.
 \begin{remark}\label{conti}
 	\rm
 	Every infinite $X$ set of cardinality $\lambda$ is a union of a continuous chain of subsets of cardinalities smaller than $\lambda$. Indeed, we 
 	can present $X$ as $\{x_i\mid i<\lambda \}$. For each $j< \lambda$ the set $X_j =\{x_i\mid i<j\}$ has smaller cardinality, and $X$ is the union of the continuous chain $(X_j)_{j<\lambda}$.
 \end{remark}
\begin{theorem}\label{T}
	In the category $ \Top$ a space is  
	\begin{itemize}
	\item	[{\rm(1)}] Finitely generated w.r.t. embeddings if, and only if, it is finite.
	\item[{\rm(2)}] Finitely generated w.r.t. open embeddings if, and only if, it is compact.
\end{itemize}
 \end{theorem}          
          
\begin{proof} 
	(1) Let $X$ be a finite space. Suppose that  $Z$ is a colimit of a directed diagram of embeddings as in Proposition \ref{dir}, then
every continuous map $f: X \to Z$ clearly factorizes in $\Set$ through some colimit map $c_i$,
and the factorization is continuous since $c_i$ is an embedding.

Conversely, if $X$ is finitely generated w.r.t. embeddings, $X$ is finite. Indeed, if $X$ is infinite, we derive a contradiction. Let $X'$ be the indiscrete space on the same set. Express it as the colimit of a continuous chain of subspaces of smaller cardinalities (using the above remark).
Then the identity map from $X$ to $X'$ does not factorize through any of the members of the chain.

(2) Let $X$ be compact and consider a directed diagram $(Z_i)_{i\in I}$ of open embeddings, which we assume to be given by subspace inclusions. Then its colimit $Z$ is the union, containing every $Z _i$ as an open subspace. Hence, any continuous map $f:X\to Z$ gives us an open cover $(f^{-1}(Z_i))_{i\in I}$ of $X$. We can choose a finite subcover $(f^{-1}(Z_i))_{i\in F}$. Choose an upper bound $j$ of $F\subseteq I$ to obtain a (necessarily unique) factorization of $f$ through $Z_j$.

Conversely, for $X$ finitely generated w.r.t. open embeddings, let $(U_k)_{k\in K}$ be an open cover of $X$. Then, for any finite set $F\subseteq K$, let $U_F=\bigcup_{k\in F}U_k$ and consider the directed system of all $U_F$. Its colimit is $X$: a set $U \subseteq X$ is open in $X$ if, and only if,
its intersection with each $U_F$ is open (since $U$ is the union of all of those intersections). Therefore the identity map on $X$ must factorize through some $U_F$. This shows that $X$ is compact.
\qed
\end{proof}

\medskip
In item (1) we have only used a continuous chain, therefore the same result characterizes finite spaces in $\Top$ as precisely the finitely \emph{small} objects w.r.t. embeddings.

\begin{corollary} \label{C:compact}
	In each of the categories $\Haus$, $\Top_1$ and $\Top_0$ a space is finitely generated w.r.t. open
	embeddings if, and only if, it is compact.
	\end{corollary}

Indeed, these subcategories are closed in $\Top$ under subspaces and under directed colimits of embeddings. Thus the proof of the above theorem works in these subcategories as well. 

We now consider directed systems of arbitrary monomorphisms in a full reflective subcategory of $\Top$.

\begin{proposition}\label{P:refl}	Let $\C$ be a full reflective subcategory of $\Top$ containing  a non-empty finitely generated space.
For every directed diagram of monomorphisms in $\C$ with colimit
	$Z$ in $\Top$, the reflection map $r_Z: Z \to RZ$ is bijective.	
\end{proposition}

\begin{proof} 
	Let $X$ be a non-empty finitely generated space. Denote by $c_i: Z_i \to Z$ the colimit of a directed diagram in $\Top$;
	its colimit in $\C$ is given by $r_Z \cdot c_i: Z_i \to RZ$.
	 To prove that $r_Z$ must be bijective, we consider an element $y$ of $RZ$ and let $f: X \to RZ$ be the constant map
	of value $y$. Since it factorizes through some $r_Z \cdot c_i$, we see that $y$ lies in the image of $r_Z$. Furthermore, since the colimit map $c_i$ is monic, the (essential) uniqueness of the factorization forces the preimage of $y$ in $Z$ to be unique.
	\qed
\end{proof}

\begin{remark}\label{R:refl}
	\rm
A completely analogous proposition holds for directed diagrams of (open) embeddings and 
	finitely generated objects w.r.t. (open) embeddings.
\end{remark}

	\medskip
	Gabriel and Ulmer proved in \cite{GU1971} that the only $\lambda$-generated object in the category of compact Hausdorff spaces is the empty space. Here is a small illustration of the above remark in the same direction:

	\begin{example}
		\rm
		 In the category of compact Hausdorff spaces the empty space is the only finitely generated object w.r.t. open embeddings.
Indeed, apply the above remark to the $\omega$-chain of discrete spaces $\{0,...,n-1\}$ for $n<\omega$ with limit $\mathbb N$ (discrete) in $\Top$: the reflection map into the Stone-$\check{\text{C}}$ech compactification $\beta(\mathbb N)$ fails to be bijective.
\end{example}

Next we provide a number of reflective subcategories of $\Top$ in which the empty space is the only finitely presentable object. 
Key to our argument is the following example.

\begin{example}\label{exam}
	\em
	We construct a directed diagram $D$ of compact, metrizable spaces whose colimit in $\Top$ is the indiscrete two-element space. The underlying directed poset of $D$ has as elements all pairs $(A,B)$ with $A$ a non-empty finite set of rational numbers in the unit interval $[0,1]$, and $B$ a non-empty finite set of irrational numbers in $[0,1]$. It is ordered by inclusion in each component, which gives a directed poset. The corresponding object of $D$ is the space $S_{A,B}$ which is the quotient of the unit interval (with its Euclidean topology) that identifies all points of $A$, and likewise for  $B$; we may write 
	$$S_{A,B}= \{A,B\}\cup([0,1]\setminus(A\cup B)).$$ 
	We have
	 the quotient maps $q_{A,B}:[0,1]\to S_{A,B}$, making the open sets of  $S_{A,B}$ be given by those open sets $U \subseteq [0,1]$ such that $A \subseteq U$ or $A \cap U= \emptyset$, analogously with $B$. The connecting maps of our diagram are the quotient maps commuting with the maps $q_{A,B}$.
	
	Clearly, $S_{A,B}$ is compact and, by finiteness of $A\cup B$, also Hausdorff. In fact, the spaces $S_{A,B}$ are even metrizable, having a countable base of open sets which, in the description above, are given by finite unions of open intervals in $[0,1]$ with rational endpoints disjoint from $A \cup B$. 
	
	The colimit of our diagram is a two-point space $C=\{\mathrm{rat},\mathrm{irr}\}$. 
	The colimit maps $s_{A,B}: S_{A,B}\to C$ are the obvious factorizations, where the inverse image of $\mathrm{rat}$ under the map $s_{A,B}\cdot q_{A,B}$ is (independently from $(A,B)$) the set of rational numbers in $[0,1]$. Since this set is neither open nor closed in $[0,1]$, the set $s_{A,B}^{-1}(\mathrm{rat})$ is not open or closed in any $S_{A,B}$, so that the point $\mathrm{rat}$ cannot be open or closed in $C$ either. Consequently, $C$ is indiscrete.
\end{example}

\begin{remark}
	\rm
	An example, more complicated than the above one, but designed for a similar reason, may be found in Dugundji's book \cite{Dugundij1966}, p. 422. He presents a directed system of `good' quotient spaces of the circle $S^1$ whose colimit is nevertheless an infinite indiscrete space. We note that the existence of directed systems of `good' spaces with poor colimits in $\Top$ has also motivated the search for sufficient conditions ensuring a better separation behaviour of the colimit; see, for example,  \cite{HajekStrecker1972}.
\end{remark}

\begin{theorem}\label{T0fp}
	Let $\mathcal C $ be a full reflective subcategory of $\Top$, contained in the category $\Top_0$ and  containing all compact metrizable spaces. Then the empty space is the only finitely presentable object.
\end{theorem}

\begin{proof} The directed diagram of the example above involves only compact metrizable spaces.
	Its colimit in $\mathcal C$ is given by the reflection of the indiscrete space $C$. Since already the T$_0$-reflection merges the two points of $C$, the colimit in $\mathcal C$ must be a singleton space. The statement thus follows from Proposition \ref{P:refl}.
\qed	
\end{proof}

				\begin{example} \label{empty}
					
				\rm
					In each of the categories $\Top_0$, $\Top_1$, $\Haus$  and $\Tych$ (of Tychonoff spaces), the only finitely presentable object
					is the empty space.
				\end{example}

	\section{\small Smallness in Haus}
We have seen an easy argument showing that the empty space is the only finitely presentable object of $\Haus$.
 Relying on a beautiful and very efficient example presented by Horst Herrlich in \cite{Herrlich1969}, several years before the appearance of \cite{GU1971}, we now show that in $\Haus$ (and a number of other important reflective subcategories) the empty space is also the lone finitely generated object w.r.t. embeddings. This corrects a mistaken claim of Example II.13(3) in \cite{AHRT2002}. 
 
 Let us first formulate Herrlich's result. Since \cite{Herrlich1969} contains no proof of it, we include ours.

\begin{example}\label{Herrlichex}
	\rm
Herrlich presented a topological space $X$ that is the colimit of an $\omega$-chain of closed subspaces, each of which is a Tychonoff space; yet, $X$ fails to be Hausdorff.

 We start with a Tychonoff space $C$ which  is not normal. Thus we can choose a pair $A,B$ of disjoint non-empty closed sets that cannot be separated by open neighbourhoods. Let us consider countably many disjoint copies $C_n,A_n,B_n (n\in \Bbb{N}$) of the given triple of spaces $C,A,B$.
 Put 
$$A^*=\bigcup_{n=0}^{\infty}A_n,\; B^*=\bigcup_{n=0}^{\infty}B_n \;\text{ and }  C^*= \bigcup_{n=0}^{\infty}C_n.$$
 
  Then, with new points $a,b$, topologize the set
$$X:=\{a,b\}\cup C^*$$
by calling $U\subseteq X$ open precisely when it satisfies the following conditions:

1. $U\cap C_n$ is open in $C_n$ for all $n$;

2. if $a\in U$, then there is some $m$ with $A_n\subseteq U$ for all $n\geq m$;

3. if $b\in U$, then there is some $m$ with $B_n\subseteq U$ for all $n\geq m$.
\end{example}

\begin{lemma} \label{proof} The space  $X$ above is a colimit of the $\omega$-chain of its closed subspaces
	$$ X_m:=\{a,b\}\cup A^*\cup B^*\cup \bigcup_{n=0}^m C_n \hspace{3mm} (m<\omega).$$
Each of them is a Tychonoff space, whereas $X$ fails to be a Hausdorff space.
\end {lemma}

\begin{proof}
One readily verifies that $X$ is a T$_1$-space, and that the following sets are closed in $X$, for all $n$:
$$C_n,\; A_n,\; B_n,\;\;
\{a\}\cup\bigcup_{i=n}^{\infty}A_i,\;\;
\{b\}\cup\bigcup_{i=n}^{\infty}B_i.$$
But $X$ is not Hausdorff: if we had open disjoint sets $U\ni a$ and $V\ni b$ in $X$, then $A_n\subseteq U$ and  $B_n\subseteq V$ for some $n$, so that $U\cap C_n$ and $V\cap C_n$ would be open disjoint neighbourhoods of respectively $A_n$ and $B_n$ in $C_n$. 

It is easy to verify that each $X_m$ is a closed subspace of $X$, and that $X$ is the colimit of the $\omega$-chain of embeddings of these
subspaces.
In contrast to the space $X$, the following sets are both closed and open in the space $X_m$, for every $m$:
$$C_n, \text{ if }n\leq m;\; A_n,\; B_n,\; \{a\}\cup\bigcup_{i=n}^{\infty}A_i,\text{ and } \{b\}\cup\bigcup_{i=n}^{\infty}B_i,\text{ if } n>m.$$
It now follows that every $X_m$ is a Tychonoff space. Indeed, considering a closed set $D$ in $X_m$ and a point $x\in X_m\setminus D$, in order to show that there is a continuous map $g:X_m\to [0,1]$ with $g(x) =0$ and $D\subseteq g^{-1}(1)$, we distinguish the following cases:

\medskip
\noindent  {\it Case} 1: $x\in C_n$, for some (and, thus, a unique) $n\leq m$: The set $C_n\cap D$ is closed in $X_m$ and, hence, also in the closed subspace $C_n$. We obtain a continuous map $f:C_n\to[0,1]$ with $f(x)=0$ and $C_n\cap D\subseteq f^{-1}(1)$ and extend $f$ to the map $g:X_m\to [0,1]$ that sends $X_m\setminus C_n$ constantly to $1$. Then, for any closed set $Z\subseteq[0,1]$, depending on whether $1\notin Z$ or $1\in Z$, the set $g^{-1}(Z)$ equals $f^{-1}(Z)$ or $f^{-1}(Z)\cup (X_m\setminus C_n)$ and is therefore closed in $X_m$ in either situation.

\medskip
\noindent {\it Case} 2: $x \in A_n$ (or, symmetrically, $x\in B_n$), for some $n>m$: The set $A_n\cap D$ is closed in $X$ and, hence, also in $C_n$, With a continuous map $f:C_n\to [0,1]$ that sends $x$ to $0$ and $A_n\cap D$ to $1$, we define $g:X_m\to[0,1]$ to coincide with $f$ on $A_n$ and to map $X_m\setminus A_n$ to $1$. For $Z$ closed in $[0,1]$, similarly as in Case 1 one obtains $g^{-1}(Z)$ to coincide with $f^{-1}(Z)\cap A_n$ or $(f^{-1}(Z)\cap A_n)\cup(X_m\setminus A_n)$,  both sets being closed in $X_m$.

\medskip
\noindent{\it Case} 3: $x=a$ (or, symmetrically, $x=b$): For the open set $(X\setminus D)\ni a$ in $X$ we have some $n>m$ with $E:=\{a\}\cup\bigcup_{i=n}^{\infty} A_i\subseteq X_m\setminus D$. Since $E$ is closed and open in $X_m$, we obtain a continuous map $g:X_m\to [0,1]$ by sending $E$ to $0$ and $X_m\setminus E$ to $1$.
\qed
\end{proof}

\begin{corollary}\label{T2fge}
	 In every full reflective subcategory $\mathcal C$ of $\Top$ contained in $\Haus$ and  containing $\Tych$, the empty space is the only finitely generated object w.r.t. embedings.
\end{corollary}	
	
Indeed, apply Remark \ref{R:refl} to the above space $X$ whose reflection in $\Haus$ is not bijective: it merges
the pair $a,b$ of elements not posesing disjoint open neighborhoods. Thus the reflection in $\mathcal C$ also merges
that pair.

\begin{remark}\label{fgTychHaus}
	\em
 Since the above example uses only an $\omega$-chain, the same corollary holds for finite smallness: the only Hausdorff space finitely small w.r.t. embeddings is the empty one.
\end{remark}

\section{\small Smallness in $\Top_0$}
We know that, in $\Top_0$, the compact spaces are the objects finitely generated w.r.t. open embeddings (Corollary \ref{C:compact}.).
We now prove that finite T$_0$-spaces are precisely the objects finitely generated w.r.t. embeddings. The proof given here is not easy,
it uses substantially Ramsey's Theorem. Afterwards we prove that the empty space is the only finitely generated object of $\Top_0$.

We use below that $\Top_0$ is closed in $\Top$ under colimits of directed diagrams of embeddings. See Corollary \ref{cor}.

\begin{theorem}\label{T0fge}
	In $\Top_0$ a space is finitely generated w.r.t. embeddings precisely when it is finite.
\end{theorem}
\begin{proof}
(1)	By the Lemma above, it follows that, like in $\Top$ (Proposition \ref{fpTop}), finite T$_0$-spaces are finitely generated w.r.t. embeddings in $\Top_0$.
	Conversely, 
	we will show that an infinite T$_0$-space $X$ cannot be finitely generated w.r.t. embeddings in $\Top_0$.
	
(2)	At first, assume that $X$ contains a strictly increasing $\omega$-chain
	$$
	X_0\subset X_1\subset\dots X_n\subset\dots
	$$
	of closed sets. Put $X_{-1}=\emptyset$ and consider the map
	$$f:X\to Y:=\omega+\{\infty\},\; x\mapsto \left\{ \begin{array}{cc} n & \mbox{if } x\in X_n\setminus X_{n-1}, \\
	\infty & \mbox{otherwise} 
	\end{array}\right\} ;$$
	it becomes continuous if we topologize $Y$ by taking its down-segments as closed sets. Observe that $Y$ is a colimit of the $\omega$-chain of its subspaces $Y_n :=\{0,\dots,n\}+\{\infty\}$. But the map $f$ does not factor through any of them. Hence, $X$ is not finitely generated w.r.t. embeddings in $\Top_0$.

	Similarly, by taking down-segments as open sets in $Y$, one obtains that $X$ containing a strictly increasing chain
	$$
	X_0\subset X_1\subset \dots X_n\subset\dots
	$$
	of open sets cannot be finitely generated w.r.t. embeddings in $\Top_0$.
	
	(3) It therefore suffices to show that every infinite T$_0$-space contains either
	a strictly increasing chain of open sets or a strictly increasing chain of closed sets. 
	
	Let $\leq$ be the specialization order on $X$, {\it i.e., }
	$x\leq y$ if, and only if, every open set containing $x$ contains $y$.
	Decompose the set $[X]^2$ of all two-element subsets of $X$ into the set 
	$Z_0$ of incomparable pairs and the set $Z_1$ of comparable ones.
	Following Ramsey's Theorem (see, e.g., \cite[3.3.7]{CK}), we have an infinite
	subset $H\subseteq X$ such that either $[H]^2\subseteq Z_0$ or
	$[H]^2\subseteq Z_1$.
	
	(4) Assume $[H]^2\subseteq Z_0$. Then there is a countable set $S=\{x_0,x_1,\dots\}$ such that the subspace $S$ is T$_1$. Thus 
	$$
	\{x_0\}\subset\{x_0,x_1\}\subset\dots
	$$
	is a strictly increasing chain of closed subsets of $S$. One therefore has a closed set $X_0\subseteq X$ such that $X_0\cap S=\{x_0\}$. There is also
	a closed set $X_1'\subseteq X$ such that $X_1'\cap S=\{x_0,x_1\}$. Put
	$X_1=X_0\cup X_1'$; then $X_0\varsubsetneqq X_1$ and $X_1\cap S=\{x_0,x_1\}$.
	Continuing this procedure, we get a strictly increasing chain of closed sets $X_0\subset X_1\subset\dots$ in $X$.
	
	(5) Finally, assume  $[H]^2\subseteq Z_1$. Then $H$ is a chain, and the topology of the subspace $H\subseteq X$ is given by up-segments as open sets. Then $H$ either contains a strictly increasing chain $x_0<x_1<\dots$ or a strictly decreasing chain $x_0>x_1>\dots$. Let $S=\{x_0,x_1\dots\}$. In the first case, we have a strictly increasing chain 
	$
	\{x_0\}\subset\{x_0,x_1\}\subset\dots
	$
	of closed subsets in $S$ and, in the same way as above, we get a strictly increasing chain of closed subsets of $X$. In the second case, we get
	a strictly increasing chain 
	$
	\{x_0\}\subset\{x_0,x_1\}\subset\dots
	$
	of open subsets in $S$ and, analogously, a strictly increasing chain of open subsets of $X$.
	\qed
\end{proof}

\medskip

\medskip

Since the proof uses only $\omega$-chains, the result above yields the same assertion 
concerning finite smallness: A space in $\Top_0$ is finitely small w.r.t. embeddings if, and only if, it is finite.
\medskip

	With the help of the following example we will be able to characterize the finitely generated objects of the category $\Top_0$.
	
	\begin{example}\label{T0example}
		\rm 
		We construct a directed system of T$_0$-spaces $Z_n,\, n\in\mathbb N$, all having the same carrier set $Z$ and being connected by identity maps, such that its colimit $Z_{\infty}$ in $\Top$ fails to be T$_0$. More detailed: $Z_{\infty}$ has two points forming an indiscrete subspace (while all other points are closed). 
		
		As the common carrier set $Z$ we take two disjoint copies of the ordered set $\mathbb N\cup\{\infty\}$, which we write as
		$$ Z= \{x_i\mid i\in \mathbb N\cup\{\infty\}\}\cup\{y_i\mid i\in \mathbb N\cup\{\infty\}\}.$$
	The following system of subsets comprise a basis of a topology $\tau_n$ on the set $Z$:
		
		\medskip
		1. $\{x_k\}, \{y_k\}$, for all $k$ with $n\leq k<\infty$;
		
		2. $U_m:=\{x_m,x_{m+1},\dots,x_{\infty}\};\,V_m:=\{y_m,y_{m+1},\dots,y_{\infty}\}$, for all  $m$ with\\ 
		\indent \quad\,$n\leq m<\infty$;
		
		3. $\{x_k\}\cup V_m,\;\{y_k\}\cup U_m$, for all $k,m$ with $k<n<m<\infty$.
		
		\medskip
		\noindent To verify that this is a basis, we show that given sets $A,B$ of the system and $z\in A\cap B$, we can find some $C\ni z$ in the system with $C\subseteq A\cap B$. Assuming $z=x_k$ (for $z=y_k$ one proceedes symmetrically), in case $n\leq k<\infty$ we take $C=\{z\}$; in case $k<n$, so that $A=\{z\}\cup V_{m_1}$ and $B=\{z\}\cup V_{m_2}$, we put $C=\{z\}\cup V_m$ with $m=\mathrm{max}\{m_1,m_2\}$; and in case $k=\infty$, we may similarly take $C= U_m$ for some sufficiently large $m<\infty$.
		
		We prove that the space $Z_n =(Z,\tau_n)$ is T$_0$. For $z\neq w$ in $Z$, we may again assume $z=x_k$. In case $n\leq k<\infty$, the point $z$ is isolated, and by symmetry we may now exclude also the cases $w= x_{\ell}$ or $w = y_{\ell}$ with $n\leq\ell<\infty$ from further consideration. In case $k=\infty$, we have $w\notin U_n\ni z$, and in case $k<n$ and $w= x_{\ell}$ or $w = y_{\ell}$ with $\ell<n$, we have $w\notin\{z\}\cup V_{n+1}$.
		
		Next we show that the topology $\tau_{n+1}$ is coarser than $\tau_n$. Let $A$ be a basic open set in $\tau_{n+1}$. If $A$ is of type 1 or 2, $A$ is clearly also open in $\tau_n$. If $A$ is of type 3, so that by symmetry we may assume $A= \{x_k\}\cup V_m$ with $k<n+1<m$, then it suffices to note that the sets $\{x_k\}$ and $V_m$ are both basic open sets of $\tau_n$, whence their union is open in $\tau_n$ as well.
		
		The colimit of the $\omega$-chain $Z_n$ is the space $Z_\infty =(Z,\tau _{\infty}$) for the finest topology $\tau_{\infty}$ coarser than every $\tau_n$. Its subspace $\{x_{\infty},y_{\infty}\}$ is indiscrete. Indeed, 
		it suffices to show that every $\tau_{\infty}$-open set $A$ containing $x_{\infty}$ contains also $y_{\infty}$. For $n=0$, as a $\tau_0$-open neighbourhood of $x_\infty$, $A$ must contain some set $U_m$ with $m>0$ and, hence, the point $x_m$. In fact, being open also in $\tau_{m+1}$, as a neighbourhood of $x_m$, $A$ must then contain some basic $\tau_{m+1}$-open set $\{x_m\}\cup V_{\ell}$ and, hence, the point $y_{\infty}$.
		
		We may leave it to the reader to check that all points $x_k, y_k \;(k<\infty)$ are closed in $Z_{\infty}$, as this property of the space will not be used in what follows.
	\end{example}
	
	\medskip
	
	\begin{corollary}\label{T0fg}
		The empty space is the only finitely generated object of the category $\Top_0$.
	\end{corollary}
	
	Indeed, in the above example the $\Top_0$-reflection map of the colimit $Z_\infty$ in $\Top$ merges $x_\infty$ with $y_\infty$. Thus we can apply Proposition \ref{P:refl}.

\medskip

\section{{\small 
		Smallness in $\Top_1$}}

We now turn to the category $\Top_1$, and prove that finitely generated objects w.r.t. embeddings are precisely the finite spaces, while the
 objects finitely small w.r.t. embeddings are precisely the countable compact spaces. Thus, in this category the difference between
 considering chains  of subspaces or directed systems of subspaces is dramatic.

\begin{remark} \label{ulra}

\rm
(1) We first take note of an obvious, but important fact: $\Top_1$ is closed in $\Top$ under colimits of directed diagrams of monomorphisms. 
In particular, the colimit maps are monic.

(2) We next recall that, given an element $x$ of a set $X$ and a free ultrafilter $\mathcal{U}$ on $X$ ({\it i.e.}, not containing finite sets),
the corresponding \emph{ultraspace} on the set $X$ has all elements but $x$ isolated, and the neighborhoods of $x$ are precisely the sets $U \cup \{x\}$ for members $U$ of $\mathcal{U}$. The ultraspaces for non-principle ultrafilters (not containing finite sets) are precisely the co-atoms of the lattice of all T$_1$-topologies on $X$: the only properly finer topology is the discrete one. Every T$_1$-topology is an infimum of co-atoms (see \cite {Fr} or \cite{Steiner}, 1.1).

\end{remark}

\begin{theorem}\label{T1fg}
A space in $\Top_1$ is finitely generated  if, and only if, it is finite.
\end{theorem}

\begin{proof}
	 Let the T$_1$-space $(X,\tau)$ be finite and thus discrete. Since it is finitely generated in $\Top$, it is finitely generated in $\Top_1$,
	by the above-mentioned closure property.
	
	Conversely, let $(X,\tau)$ be a finitely generated object of $\Top_1$. Then $\tau$ is a member of the lattice of all T$_1$-topologies on the set $X$, and we will use properties of that lattice to prove that $X$ is finite. 
	
We denote by $\mathsf M$ a set of ultraspaces on $X$ with infimum $\tau$ (see Remark \ref{ulra}). Let $\mathsf M'$ be the closure of $\mathsf M$ under finite
infima in the lattice of all T$_1$-topologies. We obtain a directed diagram of all the spaces given by $\mathsf M'$ with
$\id_X$ as connecting maps. Its colimit is  $(X, \tau)$. Since this space is finitely generated, $\mathsf M'$ is finite and $\tau$ is a member of  
$\mathsf M'$. Thus $(X, \tau)$ has only finitely many non-isolated points.

Assuming that $X$ is infinite, we derive a contradiction. The set $Y$ of all isolated points in $(X, \tau)$ is cofinite. There exists an infinite set $A \subseteq Y$ which is not a 
member of any ultrafilter in $\mathsf M'$. Indeed, if $\mathsf M'$ consists of $n$ ultraspaces, we verify this by induction on $n$. For 
$n=1$ we have a single ultrafilter $\mathcal U$. Decompose X into two infinite subsets; one of them does not lie in $\mathcal U$. For the
induction step, let $A$ have the desired property w.r.t. to all members of $ \mathsf M'$ but one, say $\mathcal U _0$; then decompose $A$
into two infinite sets and choose that subset which does not belong to $\mathcal U _0$.

The set $A$ is closed in $(X,\tau)$. Thus, for every finite set $Z \subseteq A$, we have the open subspace
$Z^*=(X \setminus A) \cup Z$ of $(X, \tau)$. Form the directed diagram of all these subspaces (with $Z$ ranging over finite subsets of $A$); its colimit is $(X,\tau)$, and the identity map does not factorize through any $Z^*$. Thus
$(X,\tau)$ is not finitely generated, a contradiction. 
\qed
	\end{proof}

\medskip
For  a characterization of T$_1$-spaces finitely generated w.r.t. embeddings, we apply the following example.

\begin{example} 
	\rm
For every infinite set $X$ we present a directed
	diagram $D$ of proper subspaces of the space $X_{\mathrm{cof}}$ with the cofinite topology
	whose colimit is $X_{\mathrm{cof}}$ (with embeddings as colimit injections).
	
	Let $\mathcal B$ be a disjoint decomposition of $X$ into an  infinite number of infinite subsets.
	We recall from \cite{Comfort} that a collection $\mathcal A$ of infinite subsets of $X$ is \emph{almost disjoint}  if 
the intersection of any distinct pair in $\mathcal A$ is finite. By Zorn's Lemma
	we can choose a maximal almost disjoint collection $\mathcal A$ containing $\mathcal B$. Without loss of generality, we may assume that $\mathcal A$ covers $X$ (if not, we switch from $X$ to the union of members of $\mathcal A$
	having the same cardinality as $X$.)

	Our diagram $D$ has as objects all spaces $Y_{\mathrm{\mathrm{cof}}}$, where
	$Y$ is a finite union of members of $\mathcal A$. Connecting morphisms are all
	embeddings between these spaces. Each such $Y=A_1 \cup...\cup A_n$  is a proper subset of $X$: indeed,
we can choose a member $B \in \mathcal B$ distinct from $A_i$ for $i=1,...,n$.
	Assuming that $Y=X$, since each intersection $B \cap A_i$ is finite, we conclude that $B$ is finite, a contradiction.
	
	To prove that $X_{\mathrm{\mathrm{cof}}}$ is the colimit of
	$D$ means to verify that
	a proper subset $P\subset X$ is finite whenever its intersection with each
	$Y$ in $D$ is closed ({\it i.e.}, either the intersection is finite, or $Y\subseteq P$).
	
	Suppose $P$ is infinite. Then, by the maximality of $\mathcal A$, the set $P\cap A$ is infinite for some $A\in\calA$,
	thus $P$ contains $A$. Take $x\in X\setminus P$. There is some $B\in\calA$ containing $x$.
	Then $P\cap(A\cup B)$ is not closed in $(A\cup B)_{\mathrm{cof}}$, which is an object of $D$, a contradiction.
		
\end{example}

	Surprisingly, $\Top_1$ does contain infinite spaces $X$ which are finitely small w.r.t. embeddings.
Recall that this means that every morphism from $X$ into a colimit of a continuous chain of subspaces 
factorizes through one of the subspaces. Indeed, essential uniqueness is automatic 
since (as already remarked above) the colimit maps are monic.

\begin{theorem}\label{T1emb}
	A space in $\Top_1$  is
	\begin{itemize}
	
	\item[{\rm(1)}] {Finitely generated w.r.t.,  embeddings if, and only if, it is finite.}
	 
	 \item[{\rm(2)}] {Finitely small w.r.t. embeddings if, and only if, it is countable and compact.}	
\end{itemize}
\end{theorem} 

\begin{proof}
(1) If  a space is finite, then by Theorem \ref{T1fg} it is finitely generated w.r.t. embeddings.
Conversely, if $X$ is finitely generated w.r.t. embeddings, it is finite, which follows from the above example.
Indeed, assume that $X$ is infinite and
 let $X_{\mathrm{cof}}$ be the space with cofinite topology on same carrier set as $X$. Then the identity map $X\to X_{\mathrm{cof}}$ cannot be factorized through any proper subspace of $X_{\mathrm{cof}}$.

(2a) Sufficiency: Every finitely small space w.r.t. embeddings is countable and compact. Indeed, an uncountable T$_1$-space $X$ 
is  not small w.r.t. embeddings because the space $X_{\mathrm{cof}}$  is a union of a continuous chain of proper 
subspaces (Remark \ref{conti}). And the  continuous map $\id: X \to X_{\mathrm{cof}}$ does not factorize through any of the subspaces. Compactness follows from Theorem \ref{T}.

(2b) Necessity:
Let $X$  be countable and compact. We prove that it is finitely small w.r.t. embeddings. Let $\alpha$ be a limit ordinal and 
$c_i: Z_i \to Z \;  (i<\alpha) $ a colimit of a continuous chain of subspaces. Given a continuous map $f: X \to Z$, we verify that it factorizes through some $c_i$. (Essential uniqueness comes for free since all $c_i$ are embeddings by Proposition \ref{dir}.)

(i) This is clear if the cofinality of $\alpha$ is uncountable:  since $X$ is countable, in $\Set$, there exists a factorization $f=c_i \cdot f'$ for some $i$.
Since $c_i \cdot f'$ is continuous and $c_i$ is an embedding, it follows that $f'$ is also continuous. 

(ii) In case $\alpha$ has countable cofinality, our chain has a cofinal subchain indexed by $\omega$. Thus we just need to prove that 
our statement holds for $\alpha = \omega$. Put $Z= \mathrm{colim} Z_n$. Assuming that $f$ does not factorize through $Z_n$ for any $n \in \omega$, we derive a contradition. Since $f$ does not factorize through $Z_0$,
we have some $x_0\in X$ and $n_0>0$ with $a_0:=f(x_0)\in Z_{n_0}\setminus Z_0$. Next $f$ does not factorize through $Z_{n_0}$, thus we obtain  some $x_1\in X$ and $n_1>n_0$ with $a_1:=f(x_1)\in Z_{n_1}\setminus Z_{n_0}$, {\em etc.} 
Continuing inductively, we get an infinite subset $A=\{a_0, a_1, \dots\}$ of $Z$ with $a_i\in Z_{n_{i+1}}\!\setminus\! Z_{n_i},\, n_{i+1}>n_i,\, i=0, 1,\dots$. By construction, all the sets $A\cap Z_n$ are finite, hence, closed in the T$_1$-space $Z_n$. This proves that $A$ is closed in the colimit $Z$. For the same reason,
every subset  $B \subseteq A$ is closed. Therefore $A$ is a discrete. The set $f^{-1}(A)$ is closed in $X$, hence it is compact. However, this implies that $A$ is compact (a continuous image of a compact set), in contradiction to being infinite and discrete.
\qed
\end{proof}

\medskip

	\section{\small Smallness in full reflective subcategories of $\Top$}	
	
	The results of the previous sections may in part be used to establish characterizations of smallness in a wide range of subcategories $\calC$ of $\Top$, beyond just $\Top$ or $\Top_i\;(i=0,1,2)$. Throughout this section, we assume that
	
	\medskip
	
	{\it $\C$ is a full reflective subcategory of $\Top$ closed under isomorphisms, and 
		
		the two-point discrete space $D$ lies in $\C$}. 
	
	\medskip
	
	\noindent Of course, as a full reflective subcategory, $\C$ is closed under retracts and products in $\Top$; in particular, the singleton space $1$ lies in $\C$. With the hypothesis $D\cong1+1\in\C$ one actually has every finite discrete space in $\C$. We do {\it not} assume that $\C$ be {\it hereditary} ({\it i.e.}, be closed under subspaces), a property which, as we recall, holds precisely when,  for all topological spaces $X$, the $\C$-reflection map $r_X:X\to RX$ is surjective, {\it i.e.}, if $\C$ is {\it epi-reflective} in $\Top$.
	
	Our discussion of smallness in $\cal C$ now branches into the consideration of two disjoint cases: $\calC$ contains a two-point indiscrete space $T$, or not. The containment of $T$ in $\calC$ is easily seen to be equivalent to the bijectivity of all reflection maps (since a reflective subcategory with all reflections monic makes these automatically epimorphic), in which case $\calC$ actually contains all indiscrete spaces. Consequently, the non-containment of $T$ in $\calC$ means equivalently that there are spaces with non-bijective reflection maps. We start the discussion with the latter case, before turning to the former.

	\subsection{\small Reflective subcategories with some non-bijective reflections}
	As noted above, our reflective subcategories $\calC$ of $\Top$ falling under this heading, while containing $D$, will {\it not} contain $T$. If $\C$ is epi-reflective and the Sierpi\'{n}ski space $S$ ({\it i.e.}, a two-point space not homeomorphic to $D$ or $T$) lies in $\calC$, then $\calC=\Top_0$; indeed, every T$_0$-space embeds into a power of $S$, and every non-T$_0$-space has an indiscrete two-element subspace. (However, we note the known fact that there are proper subcategories of $\Top_0$ containing $S$ that are epi-reflective in $\Top_0$; see \cite{Skula1969}.) If $S\notin \calC$, then $\calC$ must be contained in  the category $\Top_1$, since any non-T$_1$-space contains a subspace that is either homeomorphic to the spaces $T$ or $S$, both of which are retracts of any larger space.
	
	Our goal is to extend Corollary \ref{T0fg}, asserting that the empty space is the only finitely generated object in $\Top_0$, from that category to a very large array of reflective subcategories $\calC$ of $\Top$. By necessity, such array must exclude the category $\Top_1$, since we have identified its finitely generated objects as the finite discrete spaces (Theorem \ref{T1fg}). In what follows, we restrict our attention to subcategories of $\Top_1$, omitting non-hereditary subcategories of $\Top_0$ containing $S$.
	
	\begin{theorem}\label{fg-epi-not-bi}
		Let $\calC$ be a reflective subcategory of $\Top$ and properly contained in $\Top_1$. Then the only finitely generated object in $\C$ is the empty space, provided that $\calC$ contains every space whose topology is a finite infimum of ultrafilter topologies on its underlying set.
		
	\end{theorem} 
	
	\begin{proof} 
			According to Proposition \ref{P:refl}, we need to find a T$_1$-space $X$ with non-bijective $\C$-reflection $r_X$ since, then, $X$ cannot be an infimum of finitely many of ultraspace topologies. Indeed, as shown in the proof of Theorem \ref{T1fg}, $X$ is the colimit of a directed system of spaces whose topologies are infima of ultraspace topologies, and these spaces must all lie in $\C$, by hypothesis. 
			
			If there exists any {\it coarse T$_1$-space} $X$ ({\it i.e.}, a space $X$ carrying the cofinite topology) which lies outside $\C$, such $X$ would do the job. Indeed, $r_X$ cannot be bijective, because there is no  strictly coarser T$_1$-topology on the set $X$. 
			If, on the contrary, all coarse T$_1$-spaces lie in $\C$, then $\C$ cannot be hereditary. Indeed, every T$_1$-space embeds into a product of coarse T$_1$-spaces, so that there is some space $Y\in\C$ with $X\subseteq Y$, but $ X\notin \calC$. Hence, as an embedding that is not a homeomorphism, the reflection $r_X$, cannot be surjective.
			\qed
	\end{proof}

	 If $\calC\subseteq \Haus$, the provision of the Theorem regarding infima  may be weakened to the condition that $\calC$ contain all ultraspaces (since the infimum of finitely many ultraspaces finer than a Hausdorff space is Hausdorff and, in fact, a coproduct of finitely many ultraspaces). However, under the condition $\calC\subseteq \Haus$, we now prove a stronger assertion than that of the Theorem above, namely that the empty space is the only finitely generated space w.r.t embeddings. Our proof is based on a construction similar to Herrlich's Example \ref{Herrlichex}, which leads to a generalization of Corollary \ref{T2fge}. We preface this result by some preliminary observations.

		\begin{lemma}\label{lemma4}
			\begin{enumerate}
				\item[{\rm(1)}] If there is a discrete space not belonging to $\C$, then the empty space is the only object finitely generated w.r.t. \mbox{embeddings in $\C$.}
				\item[{\rm(2)}] If the reflective subcategory $\C$ contains all discrete spaces, then it contains all Hausdorff spaces with a finite number of accumulation points.
			\end{enumerate}
		\end{lemma}

	\begin{proof}
		(1) Consider a discrete space $X$ outside $\C$ of smallest (necessarily infinite) cardinality. Such space $X$ is a colimit in $\Top$ of its subspaces of smaller cardinality, and $r_X:X\to RX$ cannot be a bijection. Hence, Proposition \ref{P:refl} applies.
		
		(2) First, let $X$ be a Hausdorff space with a unique accumulation point $x$. 		The $\C$-reflection $r_X:X\to RX$ must be an embedding since $X$ is zero-dimensional. Assuming that $r_X$ is not bijective we have some $z\in RX\setminus X$. There is an open neighborhood $G$ of $z$ disjoint from a neighborhood $U$ of $x$. Consider the map $f:X\to X$, mapping $X\setminus U$ identically and $U$ constantly to $x$. The range of $f$ is discrete, and there is a unique continuous map $Rf:RX\to f(X)$ extending $f$. The point $z$ gets mapped by $Rf$ to a point $y\ne x$, and the preimage of $y$ is a clopen neighborhood $V$ of $z$ in $RX$ meeting $X$ in the point $y$. Now consider the map $h:RX\to RX$, mapping $RX\setminus V$ identically and $V$ constantly to $y$. Since $h$ is a retraction keeping $X$ fixed, the image $h(RX)$ is a $\C$-reflection of $X$. Therefore, the identity $X\to X$ extends uniquely to a homeomorphism $RX\to h(RX)$. Consequently, $h$ must be a homeomorphism, which implies $V=\{y\}$ and $z=y$,  contradicting the choice of $z$. Thus $X\in\calC$.
		
		Now, let the Hausdorff space $X$ have finitely many accumulation points $x_1,...,x_n$. Then there are disjoint open sets $G_i$ containing $x_i$. Thus, $X$ is the coproduct in $\Top$ of its subspaces $G_i$ (each having a unique accumulation point) and of a discrete space $X\setminus\bigcup G_i$.
		\qed
	\end{proof}
	
	\medskip

	We regard the points of $\beta(\mathbb N)$ as the ultrafilters on $\mathbb N$ and let $\mathbb{N^*}=\beta(\mathbb N)\setminus \mathbb N$. The trace of the neighborhood system of the point $\mathcal F\in\mathbb{N^*}$ in $\beta(\mathbb{N^*})$ is the filter $\mathcal F$.  
	
	\begin{lemma}\label{lemma6}
		\begin{enumerate}
			\item[{\rm(1)}] If $P$ is a dense subspace of $\mathbb{N^*}$ and $U_p\in p$ for every ultrafilter $p\in P$, then $\mathbb N \setminus (\bigcup_{p\in P} U_p)$ is a finite set.
			\item[{\rm(2)}] There exist two uncountable disjoint dense sets $P,\,Q$ in $\mathbb N{^*}$.
		\end{enumerate}
	\end{lemma}
	
	\begin{proof}
		(1) Suppose that $A=\mathbb N \setminus( \bigcup_{p\in P} U_p)$ is infinite. Then 	the closure $\overline{A}$ in $\beta(\mathbb N)$ is a clopen set meeting $\mathbb N{^*}$ in a non-void open set disjoint from $P$, which contradicts the density of $P$.
		
		(2) A base of open sets in $\mathbb N{^*}$ is formed by the sets $A^*=\overline A\setminus \mathbb N$, where $A$ is an infinite subset of $\mathbb N$. Each such set has cardinality $2^{2^\omega}$, and we may well-order them as $\{A_\alpha^*\mid\alpha\in 2^\omega\}$. Begin by picking any two distinct points $p_1,q_1$ in $A_1^*$. If, for all $\beta<\alpha$, we already have points $p_\beta,q_\beta\in A_\beta^*$, such that the sets $P_\alpha=\{p_\beta\mid\beta<\alpha\},Q_\alpha=\{q_\beta\mid\beta<\alpha\}$ are disjoint, we can find two distinct points $p_\alpha,q_\alpha$ in $A_\alpha^* \setminus (P_\alpha\cup Q_\alpha)$ since $|P_\alpha\cup Q_\alpha| < 2^\omega$. The sets $\{p_\alpha\mid\alpha\in 2^\omega\},Q=\{q_\alpha\mid\alpha\in 2^\omega\}$ are the required disjoint dense sets.
		\qed
	\end{proof} 
	
	We are ready to state our result. Proving it, we will use a finer topology on the set $\beta(\mathbb N)$, declaring every point in $\mathbb N^*=\beta(\mathbb N)\setminus \mathbb N$ to be isolated 
	and keeping its usual neighborhoods restricted to $\mathbb N$. We denote the new space with same underlying set as $\beta(\mathbb N)$ with this topology by $\delta(\mathbb N)$;  this is the so called {\it Kat\v{e}tov H-closed extension of $\mathbb N$}.
		
		\begin{theorem}
			Let the full reflective subcategory $\C$ of $\Top$ be contained in $\Haus$ and contain a two-point discrete space. Then the only space that is finitely generated w.r.t. embeddings in $\C$ is empty.
		\end{theorem}
		
		\begin{proof}
			Lemma \ref{lemma4} allows us to restrict ourselves to the case that
			all discrete spaces belong to $\C$. By Lemma \ref{lemma6} we have two disjoint dense subsets $P,Q$ of $\mathbb N^*$. Regarding $S:=\mathbb N\cup P\cup Q$ as a subspace of $\delta(\mathbb N)$, we see that $P,Q$ are discrete subspaces of $S$. 			Now we define a topology $t$ on the set
			$$X=(\mathbb N\times S)\cup\{p,q\},$$
			where $p,q\notin \mathbb N\times S$ are distinct, as follows:
			\begin{itemize}
				\item The sets $\{n\}\times S$  are open. 
				\item The sets $U_n=\{p\}\cup\bigcup_{k>n}(\{k\}\times (P\cup \mathbb N))$ are basic neighborhoods of $p$, for all $n\in\mathbb N$.
				\item The sets $V_n=\{q\}\cup\bigcup_{k>n}(\{k\}\times (Q\cup \mathbb N))$ are basic neighborhoods of $q$, for all $n\in\mathbb N$.
			\end{itemize}
			The resulting space $(X,t)$ fails to be Hausdorff, but its subspaces $$X_k:=X\setminus (\{k+1,k+2,...\}\times \mathbb N)$$ form an increasing sequence of closed Hausdorff subspaces of $X$ covering $X$. 
			
			We proceed to show that the topology $\tau$ on the set $X$ finally generated by the system $\{X_k\}_{\mathbb N}$ equals the topology $t$ of $X$. Clearly, the neighborhood systems at any point of $X$ different from $p,q$ coincide with respect to both, $t$ and $\tau$. As the situation for the remaining two points is symmetric, it suffices to show that every open set $G$ in $\tau$ containing the point $p$ is also its neighborhood in $t$. 
			
			Indeed, since $G\cap X_1$ is a neighborhood of $p$ in $X_1$, it contains the set
			$\bigcup_{n\geq k}\{n\}\times P$ for some $k\in\mathbb N$. Since every open set in $S$ containing $P$ contains $\mathbb N$ up to a finite set, $G\cap X_m$ with $m> k$ must contain $\bigcup \{\{n\}\times \Bbb N,\, n=k+1,...,m$. Consequently, $G$ contains  the $t$-neighborhood $U_n$. 
			
			Now we show that each space $X_k$ is finally generated by those of its subspaces which have only finitely many accumulation points. The space $X_k$ is a union of $\{p,q\}$, a coproduct of countably many discrete spaces, and of finitely many copies of $S$. The points $p,q$ have  countable bases of neighborhoods. The space $S$ is finally generated by subspaces which are ultraspaces. 
			
			We denote by $\calP$ (or $\calQ$) the set of all sequences from $\bigcup_{\mathbb N}\{n\}\times P$ (or from $\bigcup_{\mathbb N}\{n\}\times Q$) converging to $p$ (or to $q$, resp.). The system $\calS$ of all subspaces of $S$ with a finite number of accumulation points finally generates $S$, and we may assume that $\calS$ covers $S$. The system $\calP\cup\calQ\cup \bigcup_{\mathbb N}(\{n\}\times \calS)\cup \bigcup_{n\le k}(\{n\}\times \mathbb N)$ finally generates $X_k$, as well as the system of unions of finitely many members of $\calS_k$. We have $\calS_k\subseteq \calS_m$ for $k\le m$. Consequently, a set $G$ is open in $X$ if, and only if, the intersections with every $X_k$ are open; equivalently, if, the intersections with every member of $\mathcal{S}_k$ are open. 
			
			We exhibited a directed system of embeddings of spaces from $\calC$ whose colimit is $(X,t)$. Since the reflection $X\to RX$ is not bijective, we conclude that only the empty space is finitely generated w.r.t. embeddings in $\C$.
		\end{proof}  \qed

		\subsection{\small Reflective subcategories with all reflections bijective}
		
		Our reflective subcategory $\C$ of $\Top$ under consideration now contains the two-point indiscrete space $T$ (and, hence, all indiscrete spaces) and, equivalently, makes the reflection maps $r_X:X\to RX$ bijective, for all $X\in\Top$. The largest such subcategory $\C$ is, of course, $\Top$ itself, for which we identified the finitely generated objects as the finite discrete spaces in Proposition \ref{fpTop}. The least such subcategory $\C$ contains precisely the indiscrete spaces, which is isomorphic to $\Set$, so that the finitely generated objects are trivially characterized as {\it all} the finite objects in that category. We are about to show that for any other $\C$, the characterization of finitely generated objects generalizes the case $\C=\Top$.
		
		It is easy to show that $\C$, having bijective reflection maps, must contain every space $X$ that is the domain of some initial source $f_i:X\to Y_i$ in $\Top$ with all $Y_i\in\C,\;i\in I$. In particular, considering for any space $X$ the source of all continuous maps $X\to S$, one sees that $\C$ must equal $\Top$, as soon as the Sierpi\'nski space $S$ is in $\C$. Consequently,
		if $S\notin\C$, then
		every space $X$ in $\C$ must be {\it symmetric}, that is: the specialization order of $X$ must be symmetric, so that $x\in\overline{\{y\}}$ implies $y\in\overline{\{x\}}$ for all $x,y\in X$. In other words, the subcategory of symmetric spaces is largest amongst all proper subcategories $\C$ of $\Top$ under consideration in this subsection.
		
		There is also a least member amongst all subcategories $\C$  under consideration which contains not just indiscrete spaces. Indeed, any epi-reflective $\C$ containing $T$ and at least one non-indiscrete space must contain all zero-dimensional spaces.
		
		\begin{theorem}
			If the reflective subcategory $\calC$ of $\Top$ contains all indiscrete spaces and at least one non-indiscrete space, then its finitely generated objects are precisely the finite discrete spaces.
		\end{theorem}
		
		\begin{proof}
			Under minimal modifications one can use the same argumentation given for $\Top$ also for $\C$ and obtain that all finite discrete spaces are finitely generated in $\calC$, and that conversely all finitely generated objects of $\C$ must be finite. The only remaining issue is therefore for us to show that these spaces must be discrete. Since both $T,S$ are retracts of any larger space and finitely generated spaces are closed under retracts, one must show that neither $T$ nor $S$ is finitely generated.

				Take two disjoint countable dense subsets $X_0,X_1$ of the  interval $[0,1]$ with $0\in X_0, 1\in X_1$. For every $n\geq 1$ we provide the set $X:=X_0\cup X_1$ with the following topology $\tau_n$:
				\begin{itemize}
					\item basic neighborhoods of any point of $X_0\cap [0,1-1/n]$ are the sets\\ 
					$(X_0\cap [0,1-1/n])\cup (X\cap [0,1/k)),k\in\mathbb N$;
					\item basic neighborhoods of any point of $X_1\cap[1/n,1)$ are the sets\\ 
					$(X_1\cap[1/n,1])\cup (X\cap (1-1/k,1],k\in\mathbb N$;
					\item neighborhoods of the remaining points are the usual Euclidean neighborhoods restricted to $X$. 
				\end{itemize}
				The topologies $\tau_n$ are zero-dimensional and, thus, $(X,\tau_n)\in\calC$. In the finest topology $\tau$ coarser than all $\tau_n$, the subspace $\{0,1\}$ is indiscrete, while it is discrete in all $\tau_n$. Consequently, neither $T$ nor $S$ are finitely generated in $\calC$.
				\qed
		\end{proof}

		For the characterization of the finitely generated spaces w.r.t. embeddings in our subcategory $\C$, we just note that the proof of Theorem 3.3 remains intact; consequently:
		
		\begin{corollary}
			If the reflective subcategory $\calC$ of $\Top$ contains all indiscrete spaces, then the objects finitely generated  w.r.t. embeddings in $\C$ are precisely the finite spaces in $\calC$.
		\end{corollary}

\begin{remark}	
\rm		In this section we considered reflective categories in $\Top$ containing a two-point discrete space. That is a natural assumption. Nevertheless, there are reflective categories in $\Top$ containing a space with more than one point but no such space is discrete. Horst Herrlich constructed a very nice example of such a subcategory in \cite{Herrlich1969a}. The category is composed of all powers of a strongly rigid space having more than one point. A natural question is what the finitely generated spaces in that category are. 
We can show that, again, only the empty space is finitely generated, even w.r.t. embeddings, as follows.
	
	Let $X$ be a strongly rigid Hausdorff space and $\calC$ the full subcategory of $\Top$ formed by all powers of $X$. For a cardinal $\lambda$, let $Z_\lambda$ be the reflection  into $\calC$ of a discrete space $D_\lambda$ of cardinality $\lambda$. By induction, define $\mu_1=|X|,\; \mu_{n+1}=2^{\mu_n},\;n\in\mathbb N$ and $\mu=\sup \mu_n$. For $\lambda<\nu$ the set $D_\lambda$ is a retract of $D_\nu$, which implies that $Z_\lambda$ is a retract of $Z_\nu$, thus homeomorphic to a closed subspace of $Z_\nu$. Consequently, $(Z_{\mu_n})_{n\in\mathbb N}$ is an increasing sequence of spaces in $\calC$, where each $Z_{\mu_n}$ may be considered as a closed subspace of $Z_\mu$. Let $Z=\bigcup Z_{\mu_n}$ be endowed with the final topology generated by $(Z_{\mu_n})$, so that $Z$ is the colimit in $\Top$ of a directed limit of embeddings in $\calC$. 
	
	We show $Z\notin\calC$. Every member $X^\lambda$ of $\calC$ has cardinality equal to either $2^\lambda$, provided that $\lambda\ge |X|$, or of at most $2^{|X|}$ otherwise. Since $|Z|=\mu> 2^{|X|}$, any possible $\lambda$ with $X^\lambda\cong Z$ must be bigger than $|X|$. If $\lambda<\mu$, then $\lambda<\mu_n$ for some $n$ and, thus $|X|^\lambda\le 2^{\mu_n}<\mu=|Z|$. If $\lambda\ge\mu$, then $|X|^\lambda\ge 2^\mu>\mu=|Z|$. Thus the space $Z$ can never be homeomorphic to $X^\lambda$ for some $\lambda$.  Consequently, $|RZ|>|Z|$, whence the reflection $Z\to RZ$ is not bijective, so that no non-empty space is finitely generated w.r.t. embeddings in $\calC$. 
	\end{remark}

\section{\small Further properties of subcategories of $\Top$}

\subsection{\small Where does `finitely generated' just mean `finite'?}
We have seen that the meaning of `finitely generated' may change drastically, depending on the the category of consideration: in $\Top$ it means 'finite and discrete', and in $\Top_0$ it just means `empty'.
Here is a coreflective subcategory of $\Top$, containing not just discrete spaces, where `finitely generated' plainly means that the carrier set is finite, as we have seen it for the reflective subcategory $\Top_1$.

\begin{proposition}
	The finite topological spaces are the finitely generated objects in the category $\mathsf{Alex}$ of Alexandroff-discrete spaces (where arbitrary intersections of opens are open). The finite T$_0$-spaces are the finitely generated objects of the category $\mathsf{Alex}_0$ of Alexandroff-discrete T$_0$-spaces.
\end{proposition}
\begin{proof}
	The category $\mathsf{Alex}$ is the image of the full coreflective embedding $A:\mathsf{Ord}\to\Top$ which provides a preordered set with the topology of down-closed sets as opens. Its right adjoint provides a topological space with the (dual of the) so-called specialization preorder. The image of the restriction of $A$ to the category $\mathsf{Pos}$ of partially ordered sets is the category $\mathsf{Alex}_0$. In $\mathsf{Pos}$ the finitely generated objects are precisely the finite posets (see \cite{AR1994}, 1.2(3)). The straightforward proof of this fact works more generally for preordered sets.
	\qed
\end{proof}

\subsection{\small Are there non-trivial  locally presentable subcategories of $\Top$?}

The short answer is affirmative: Theorem 3.4 of \cite{FajstrupRosicky2008} confirms the existence of a non-trivial {\it locally presentable} (see  %
\cite{GU1971, MP1989, Borceux1994, AR1994}) full subcategory, not only in the category $\Top$, but in any fibre-small concrete topological category $\mathcal K$; the authors use model-theoretic methods to show, that for every (small) set
 $\mathcal G$ 
 of objects in such a category, the full subcategory of those objects $X$ which are the target of a $\mathcal K$-final sink of arrows with domain in $\mathcal G$, is not just coreflective (see \cite{AHS1990}), but locally presentable. For instance, in $\Top$ the Alexandroff-discrete spaces come about in this  way when one takes for $\mathcal G$ all (up to homeomorphism) finite spaces.

However, with the characterization of their $\lambda$-presentable (or at least finitely presentable) objects for the familiar categories of spaces considered in this paper, it of course follows, that they cannot be locally $\lambda$-presentable (or finitely presentable, respectively). But there is often an easier way of arguing, without referral to such a characterization. Indeed, since local presentability of a category entails the existence of a strong generator (so that every object must be the codomain of a strong epimorphism whose domain is a coproduct of objects of the generator), many familiar subcategories of $\Top$ are easily seen not to be locally presentable, not even {\it nearly locally presentable} in the sense of \cite{PR2018}. Here is a (probably known) non-existence proof for a strong generator for some of the categories considered in this paper.

\begin{proposition}\label{HusekProposition}
  None of the categories $\Top$ or $\Top_i$ for $i=0,1,2$, has a strong 
    generator.
 \end{proposition}
 
 \begin{proof} A set $\mathcal G$ of objects is a strong generator in $\Top$ if, and only if, for every 
    object $Y$ the sink of morphisms with domain in $\mathcal G$ and codomain $Y$ is final; that is:
    a subset of $Y$ is open if its preimage under every continuous map $X\to Y$ with $X\in\mathcal G$ is open in $X$. 
    
    Assuming the existence of such a set $\mathcal G$, we give a space $Y^*$ violating the finality condition: for a discrete space $Y$ of cardinality greater than the cardinality of every space in $\mathcal G$, topologize $Y^*= Y\cup \{t\},\, t\notin Y$, by taking as
   open neighbourhoods of $t$ precisely the sets $\{t\}\cup(Y\setminus Z)$ with $Z\subseteq Y,\, \mathrm{card} Z < \mathrm{card} Y$. 
   Then,  
    for every continuous map $f: X \to Y^*$
    with $X\in \mathcal G$, the preimage of $t$ is open: it is either empty, or equal to 
    the preimage of the open set $\{t\} \cup (Y\setminus (f[X]\setminus\{t\}))$, whence open in $X$ in either case, although $\{t\}$ is not open in $Y^*$.

   This concludes the non-existence proof for a strong generator in $\Top$.
   Concerning its subcategories $\Top_i,\, i=0,1,2$, it suffices to note that the space $Y^*$, as a regular T$_1$ space, actually lies in the category $\Top_i$, in which monomorphisms are characterized as in $\Top$, so that one may argue for the non-existence of a strong generator just as in the case of $\Top$.
 \qed   
       \end{proof}
    
    \medskip
 Note that the one-point space is a strong generator of the monadic category of compact Hausdorff spaces over $\Set$ but, not being ranked, this category nevertheless fails to be locally presentable (see \cite{GU1971}). A more general reason for this failure is the local presentability of its dual category, which is equivalent to the category of commutative C$^*$-algebras (for their universal-algebraic description, see \cite{WPR1993}); indeed, the dual of a locally presentable category is never locally presentable, unless it is equivalent to a preorder (\cite{GU1971}, Satz 7.13). 
 
 This fact has encouraged the search for further reflective subcategories of $\Top$ whose dual is locally presentable, or at least nearly locally presentable. For example, assuming the Vop$\check{\mathrm e}$nka Principle (VP), Hu$\check{\text{s}}$ek and Rosick\'{y} \cite{HR2018} show that the dual of the category of real-compact spaces is locally presentable; in the absence of the very strong set-theoretical hypothesis (VP), they also prove that this dual category is nearly locally representable if, and only if, measurable cardinals exist.
 
 The notion of {\it locally ranked} category (as used  in \cite{AHRT2002} to prove a generalized version of Quillen's `Small Object Argument') does not require the presence of a strong generator, and it therefore makes it easier for a full subcategory of $\Top$ to fit into this generalization of the notion of locally presentable category.

\bigskip

\noindent J. Ad\'{a}mek\newline
Czech Technical University\newline
Prague, Czech Republic\newline
and\newline
Technical University Braunschweig\newline
Braunschweig, Germany\newline
j.adamek@tu-bs.de

\medskip
\noindent M. Hu\v{s}ek\newline
Faculty of Mathematics and Physics\newline
Charles University\newline
Prague, Czech Republic\newline
mhusek@karlin.mff.cuni.cz

\medskip
\noindent J. Rosick\'{y}\newline
Department of Mathematics and Statistics,\newline
Masaryk University, Faculty of Sciences,\newline
Kotl\'{a}\v{r}sk\'{a} 2, 611 37 Brno,
Czech Republic\newline
rosicky@math.muni.cz

\medskip
\noindent W. Tholen\newline
Department of Mathematics and Statistics\newline
York University\newline
Toronto ON, M3J 1P3, Canada\newline
tholen@yorku.ca


\begin{thebibliography}{99}
\bibitem{mathoverflow}{mathoverflow: mathoverflow.net/questions/288648/what-are-compact-objects-in-the-category-of-topological-spaces }
\end{thebibliography}
\end{document}